\documentclass[11pt]{article}
\usepackage{amssymb,amsmath,amsthm}
\usepackage{graphicx}

\theoremstyle{plain}
\newtheorem{theorem}{Theorem}[section]
\newtheorem{proposition}[theorem]{Proposition} 
\newtheorem{corollary}{Corollary}
\newtheorem*{corollary*}{Corollary}

\theoremstyle{definition}
\newtheorem*{remark*}{Remark}

\numberwithin{equation}{section}

\newcommand{\RR}{\mathbb{R}}
\newcommand{\ZZ}{\mathbb{Z}}

\newcommand{\eps}{\varepsilon}
\newcommand{\sign}{\mathrm{sign}\,}
\newcommand{\arccot}{\mathop\mathrm{arccot}}

\newcommand{\const}{\mathrm{const}}
\newcommand{\si}{\mathrm{si}}
\newcommand{\Si}{\mathrm{Si}}
\newcommand{\Ei}{\mathrm{E}}
\newcommand{\Ci}{\mathrm{Ci}}
\newcommand{\Cin}{\mathrm{Cin}}

\renewcommand{\Re}{\mathrm{Re}}
\renewcommand{\Im}{\mathrm{Im}}

\newcommand{\arcsh}{\mathop\mathrm{arcsinh}}

\newcommand{\Ll}{\mathrm{L}}
\newcommand{\Sl}{\mathrm{S}}
\newcommand{\Tl}{\mathrm{T}}

\newcommand{\Spi}{\mathrm{\phi}} 
\newcommand{\reE}{\mathrm{g}} 
\newcommand{\imE}{\mathrm{f}}

\newcommand{\Ssi}{\mathrm{Ssi}}
\newcommand{\Sci}{\mathrm{Sci}}
\newcommand{\Cci}{\mathrm{Cci}}

\newcommand{\Eci}{\mathrm{Eci}}

\newcommand{\ba}[1]{\begin{array}{#1}}
\newcommand{\ea}{\end{array}}
\newcommand{\dst}{\displaystyle}

\begin{document}


\title%
{Around the Fej\'{e}r-Jackson inequality: 
Tight bounds 
for certain oscillatory functions
via Laplace transform representations}

\author{Sergey Sadov}
\date{}

\maketitle

\begin{abstract}
The error of approximation of the $2\pi$-periodic sawtooth function $(\pi-x)/2$, $0\leq x<2\pi$, 
by its $n$-th Fourier polynomial is shown to be bounded by $\arccot((2n+1)\sin(x/2))$.
Related, asymptotically tight inequalities with explicit constants 
are given for the integral of the Dirichlet kernel interpolated to
non-integer values of frequency parameter
and for the Taylor series remainder of the logarithmic function $\log(1-z)$ in the unit circle.
The proofs are based on the Laplace transform representation of the Lerch Zeta function with
$s=1$.
 
\end{abstract}



\section{Introduction}

In this paper we give a unified treatment, based on the Laplace transform representations, of
some inequalities for specific sums and integrals involving trigonometric functions.

Our guiding principle is to exhibit non-oscillating envelopes of graphs of oscillatory functions
that are asymptotically tight as frequency parameter becomes large, and do not contain hidden $O$-constants.

This is in a guise the same theme as an investigation of the Gibbs phenomenon in Fourier theory.


Here are two characteristic results.

\begin{theorem}
\label{thm:myFJT}
For any integer $n\geq 1$ and $0< x<\pi$,
\begin{equation}
\label{ineq:myFJT}
\left|\sum_{k=1}^n\frac{\sin kx}{k}-\frac{\pi-x}{2}\right|<\arccot\left((2n+1)\sin\frac{x}{2}\right).
\end{equation}
\end{theorem}

Theorem~\ref{thm:myFJT} is an improvement (outside of an $O(1/n)$-neighborhood of $x=\pi$)
of the classical Fej\'{e}r-Jackson-Tur\'{a}n inequality. It has served as the main motivation
for this research. Details and discussion are given in Section~\ref{sec:fjt}.
We suggest the reader to look at the graph, Fig.~\ref{fig:fejer-jackson-all-minorants-n10} in that section.

\begin{theorem}
\label{thm:intD-sumsquares}
For any real $\lambda$ and $0<y<\pi/2$
\begin{equation}
\label{ineq:intD-sumsquares}
\left(\lambda\int_0^y \frac{\cos \lambda t}{\cos t}\,dt\right)^2+
\left(\lambda\int_0^y \frac{\sin \lambda t}{\cos t}\,dt-1\right)^2<\frac{1}{\cos^2 y}.
\end{equation}
\end{theorem}

Again, a picture, Fig.~\ref{fig:trig-ineq} in Subsection~\ref{ssec:trig_ineq}, gives a good idea of the quality
of the estimate \eqref{ineq:intD-sumsquares}.

Theorems~\ref{thm:myFJT} and \ref{thm:intD-sumsquares} are closely related, reason being, roughly speaking,
that the Fourier series remainder for the sawtooth  function in the left-hand side of \eqref{ineq:myFJT}
is expressed as an integral of the $n$-th Dirichlet kernel. The left-hand side of \eqref{ineq:intD-sumsquares} 
contains similar integrals with real frequency parameter $\lambda$, which plays the same role as the odd integer 
$2n+1$ in the $n$-th Dirichlet kernel.

\smallskip
The main body of the paper deals with various estimates for the Fourier series of the form
$$
 \sum_{k=1}^\infty \frac{e^{ikx}}{k+\mu},
$$
which are a particular case of the Lerch Zeta function $\zeta(s, e^{x/2\pi},\mu)$ with $s=1$.
As a convenient label, we use the term {\em Fej\'{e}r-Jackson series}\ (FJ series, for short) to  
refer to functions represented by infinite or finite series of this kind. 

Estimates for the FJ sums are based in most part on a Laplace transform representation.
This approach is by no means new but it proved to be very efficient for the current purposes, particularly,
through the use of the comparison monotone function 
$M(t)$ which is the Laplace transform of $(1+s^2)^{-1/2}$.

\smallskip
The final Section~\ref{sec:LogTaylor} presents another application of our general inequalities:
an explicit (no hidden constants) error estimate of $n$-th Taylor approximation of $\log(1-z)$ with uniform $O(|z|^{n+1}|\log|1-z||)$ behaviour in the unit circle.

\smallskip
Our subject, of course, has deep roots in classical literature. A brief review on material
directly related to Theorem~\ref{thm:myFJT} is given in Section~\ref{sec:fjt}. Quite a few papers 
and books cited therein contain excellent survey sections, which reveal various connections and can be consulted by the
interested reader.

Other than in Section~\ref{sec:fjt}, references are given only when they are directly relevant to particular technical elements.

Needless to say, in a text on a classical topic like this one, no claim of novelty can be safely made with respect to any 
single formula. Many auxiliary results, even those called a Proposition or a Theorem, are easier to prove than
to locate a definitive reference. In any case, to the best of my knowledge, the main results --- Theorems~\ref{thm:myFJT},
\ref{thm:intD-sumsquares}, and \ref{thm:Logtaylor-simple}, \ref{thm:Logtaylor:log1-z} --- are not found in the
prior publications.  

Numerical experimentation, naturally, played a significant role in research leading to this paper. 
The classic authors of pre-computer age deserve our apology as they were greatly disadvantaged in this respect.

\section{Preliminaries on one-variable functions}

\subsection{Sine, cosine, and exponential integrals}

Let us recall the standard special functions \cite[Sec.~5.2]{Abramowitz-Stegun1964}:
the cosine integral and the regularized cosine integral
\begin{equation}
\label{ci}
 \Ci(t)=-\int_t^\infty\frac{\cos u}{u}\,du,
 \qquad
 \Cin(t)=\int_0^t\frac{1-\cos u}{u}\,du
 =\log t-\Ci(t)+\gamma
 , 
\end{equation}
the sine integral and the complementary sine integral 
\begin{equation}
\label{si}
 \Si(t)=\int_0^t\frac{\sin u}{u}
 \,du,
\qquad
 \si(t)=-\int_t^\infty\frac{\sin u}{u}\,du=\Si(t)-\frac{\pi}{2},
\end{equation}
and the exponential integral of an imaginary argument
\begin{equation}
\label{ei}
 \Ei_1(-it)=\int_t^\infty \frac{e^{iu}}{u}\,du=-\Ci(t)-i\,\si(t).
\end{equation}
The most convenient for our purposes will be the variant of the exponential integral 
$$
E(t)=\reE(t)+i \imE(t), 
$$
with
real and imaginary parts called the {auxiliary trigonometric integrals}
\begin{equation}
\label{aux-fg}
\ba{ l}\dst
\reE(t)=-\cos t\;\Ci(t)-\sin t\;\si(t), 
\\[0.5ex]\dst
\imE(t)\,=\;\sin t\;\Ci(t)-\cos t\;\si(t).
\ea
\end{equation}
Note that
\begin{equation}
\label{E-to-si}
\si(t)=\Re(i e^{it} E(t)).
\end{equation}

The function $E(t)$ has the integral representation
\begin{equation}
\label{E}
 \ba{rcl}
 E(t)&=& 
 \dst
 \int_t^\infty \frac{e^{i(u-t)}}{u}\,du=
 \int_0^\infty \frac{e^{iu}}{u+t}\,du=
 \int_0^\infty \frac{e^{iut}}{u+1}\,du.
 \ea
\end{equation}

\begin{proposition}
\label{prop:E}
The function $E(t)$ admits the Laplace transform representation
\begin{equation}
\label{lt:E}
 E(t)=\int_0^\infty \frac{e^{-tu}}{u-i}\,du.
\end{equation}
\end{proposition}

\begin{proof}
By Watson's lemma, the integral of the analytic function $z\mapsto e^{-tz}/(z-i)$ over the arc $|z|=R$, $-\pi/2<\arg z<0$, is $o(1)$
as $R\to\infty$. Therefore the integration path $(0,+\infty)$
in the right-hand side of \eqref{lt:E} can be changed into $(0,-i\infty)$
and we get 
$$
 \int_0^\infty \frac{e^{-tu}\,du}{u-i}=
\int_0^\infty \frac{e^{itv}\,(-i\,dv)}{-iv-i}=
\int_0^\infty \frac{e^{itv}\,dv}{v+1}. 
$$ 
\end{proof}

\subsection{The comparison function $M$}

The real-valued function
\begin{equation}
\label{M}
 M(t)=\int_0^\infty \frac{e^{-tu}\,du}{|u-i|}=\int_0^\infty \frac{e^{-tu}\,du}{\sqrt{u^2+1}}, \qquad t>0,
\end{equation}
which majorizes $|E(t)|$ 
will play a prominent role in the sequel.

The function $M(t)$ satisfies the nonhomogeneous Bessel's differential equation of order zero, $$tM''+M'+tM=1,$$ and 
can be expressed in terms of the Bessel function of the 2nd kind and the Struve function \cite[(12.1.8)]{Abramowitz-Stegun1964}:
$$
M(t)=\frac{\pi}{2}(\mathrm{H}_0(t)-Y_0(t)). 
$$

It is immedately clear from the defining formula that $M$ is decreasing in $\RR_+$. It helps to  
know its asymptotic behaviour:
\begin{equation}
\label{asym:Mt-infty}
 M(t)\sim\frac{1}{t}-\frac{1}{t^3}+\frac{9}{t^5}+O(t^{-7})
 \quad\text{as $t\to\infty$},
\end{equation}
(following easily from the integral representation) and
\begin{equation}
\label{asym:Mt-0}
 M(t)=\ln t+C_1+t+O(t^2),
\quad\text{as $t\to 0^+$},
\end{equation}
where $C_1=\ln 2-\gamma\approx 0.116$. 

The asymptotics \eqref{asym:Mt-0} is a consequence of the known asymptotics in the theory of Bessel functions
or can be derived directly from the integral.

\subsection{Inequalities for one-variable functions}

\begin{proposition}
\label{prop:M-1}
The function $1/t-M(t)$ is positive and decreasing.
\end{proposition}

\begin{proof}
This follows from the Laplace transform representation with positive density
$$
\frac{1}{t}-M(t)=\int_0^\infty e^{-tu}\left(1-\frac{1}{\sqrt{1+u^2}}\right)\,du.
$$
\end{proof}

\begin{proposition}
\label{prop:E-si-acot-M}
(a) For all $t>0$ there holds the inequality 
\begin{equation}
\label{ineq:E-M}
|E(t)|<M(t).
\end{equation}

\noindent
(b) There exist $t_0$ and $t_1$, $0<t_1<t_0<1$, such that for all $t>t_0$
\begin{equation}
\label{ineq:M-acot}
 M(t)<\arccot t
\end{equation}
and for all $t>t_1$
\begin{equation}
\label{ineq:E-acot}
 |E(t)|<\arccot t.
\end{equation}
Numerically, the best values of $t_0$ and $t_1$ are $t_0\approx 0.7095667635$ and
$t_1\approx 0.4685633187$.

\smallskip\noindent
(c) For all $t\geq 0$
\begin{equation}
\label{ineq:si-acot}
|\si(t)|\leq \arccot t
\end{equation}
with equality only for $t=0$.
\end{proposition}

\begin{proof}
(a) Obvious by comparing the Laplace integrals \eqref{M} and \eqref{lt:E}. 

(b) 
Integrating the right-hand side of \eqref{M} by parts three times we get
$$
%
 M(t)=\frac{1}{t}-\frac{1}{t^3}+\frac{1}{t^3}
 \int_0^\infty e^{-tu}\frac{3u(3-2u^2)}{(u^2+1)^{7/2}}\,du
<\frac{1}{t}-\frac{1}{t^3}+\frac{9}{t^5}.
$$
By Taylor's expansion of arctangent,
$$
 \arccot t>\frac{1}{t}-\frac{1}{3t^3}
$$
for $t\geq 1$, hence $M(t)<\arccot t$ at least for (say) $t\geq 4$. For $t\in [1,4]$
the inequality \eqref{ineq:M-acot} is verified by interval analysis (numerical evaluations on a finite grid
with controlled error). 

The existence of $t_0\in(0,4)$ for \eqref{ineq:M-acot} and $t_1\in (0, t_0)$ for \eqref{ineq:E-acot} is obvious.

The best values of $t_0$ (resp. $t_1$) are found numerically as the unique 
positive root of the equation $M(t)=\arccot t$ (resp. $|E(t)|=\arccot t$).

\smallskip
(c) Since $|\si(t)|\leq |E(t)|$, in view of (a) and (b) it remains to check the inequality
\eqref{ineq:si-acot} for $0\leq t \leq 1$ (say). In this interval $\si(t)<0$ and
$$
 \arccot t-|\si(t)|=\int_0^t\left(\frac{\sin u}{u}-\frac{1}{1+u^2}\right)\,du.
$$
Positivity of the integrand follows from the estimates
$\sin u>u-u^3/6$ and $(1+u^2)(1-u^2/2)<1$ valid for $0< u<1$. 
\end{proof}

For small values of $t$ the estimate \eqref{ineq:M-acot} diverges as $t^{-1}$,
which is worse than the asymptotics \eqref{asym:Mt-0}. In the next theorem we give
a non-asymptotic estimate of the true order of magnitude as $t\to 0^+$.
The case of small $t$ will be used, via part (c) of Theorem~\ref{thm:FJ-L-M}, in the proof
of Theorem~\ref{thm:Logtaylor:log1-z}.

\begin{proposition}
\label{prop:ubM-small-t}
For any $t>0$ the inequality 
\begin{equation}
\label{ineq:M-arcsh}
 M(t)<\arcsh\frac{1}{t}
\end{equation}
holds.
Therefore for $0<t<1$ the inequality 
$$
 M(t)<|\ln t|+C_2
$$
 with $C_2=\ln(1+\sqrt{2})<0.8814$ holds.
\end{proposition}

\begin{proof}
Integration by parts leads to the identity
$$
 tM(t)=1+
 \int_0^\infty e^{-ut} d(u^2+1)^{-1/2}.
$$
Using the inequality $e^{-ut}>1-ut$ for $0<u<1/t$, we get
$$
\begin{array}{rcl}
\displaystyle
-\int_0^\infty e^{-ut} d(u^2+1)^{-1/2}
&\geq &
\displaystyle -\left.\frac{1-tu}{\sqrt{1+u^2}}\right|_0^{1/t}
+\int_0^{1/t} \frac{t\,du}{\sqrt{1+u^2}}
\\[3ex]
&=& \displaystyle
1-\int_0^{1/t}\frac{t\,du}{\sqrt{1+u^2}}
\end{array}
$$
Therefore
$$
 M(t)<\int_0^{1/t}\frac{du}{\sqrt{1+u^2}}
 =\arcsh\frac{1}{t},
$$
and we get the inequality \eqref{ineq:M-arcsh}.

Finally, 
$$
 (\arcsh x-\ln x)'=\frac{1}{\sqrt{1+x^2}}-\frac{1}{x}<0. 
$$
Hence, if $t>1$, then $\arcsh(1/t)-\ln(1/t)<\arcsh 1=\ln(1+\sqrt{2})$.
\end{proof}

\section{Sums of Fej\'{e}r-Jackson type}
\label{sec:FJsums}

\subsection{Definition of the FJ sums}
\label{ssec:def-FJsums}

First, we define the {\em 
exponential FJ sums}
\begin{equation}
\label{Lmu}
 L(x,\mu)=\sum_{k=1}^\infty \frac{e^{ikx}}{k+\mu},\qquad \mu\neq-1,-2,\dots,
\end{equation}
and the trunctated versions 
\begin{equation}
\label{Lmun}
 L_n(x,\mu)=\sum_{k=1}^n \frac{e^{ikx}}{k+\mu}.
\end{equation}

Note that
\begin{equation}
\label{L-S}
 L_n(x,\mu)=L(x,\mu)-e^{ix n}L(x,\mu+n).
\end{equation}

It will be convenient also to use a variant of this function with different normalization of the arguments,
\begin{equation}
\label{Llam}
 \Ll(x,\lambda)=e^{-ix}\,L\left(2x,\frac{\lambda-1}{2}\right)=2\sum_{k=1}^\infty
 \frac{e^{(2k-1)ix}}{2k+\lambda-1}.
\end{equation}

The real and imaginary parts in \eqref{Lmu} are the {\em cosine}\ and {\em sine FJ sums},
$$
 T(x,\mu)+iS(x,\mu)=\sum_{k=1}^\infty \frac{\cos kx}{k+\mu}
 +i\sum_{k=1}^\infty \frac{\sin kx}{k+\mu}.
$$
The finite sums $T_n(x,\mu)$ and $S_n(x,\mu)$ are defined similarly.

Doing the same for Eq.~\eqref{Llam} we have
$$
 \Tl(x,\lambda)+i\Sl(x,\lambda)=2\sum_{k=1}^\infty \frac{\cos (2k-1)x}{2k-1+\lambda}
 +2i\sum_{k=1}^\infty \frac{\sin (2k-1)x}{2k-1+\lambda}.
$$

\smallskip
Note the well known evaluations  ($-\pi<x<\pi$)
\begin{equation}
\label{evalSaw}
L(x,0)=
\,-\log\left|2\sin\frac{x}{2}\right|\,+\,i\frac{\pi-x}{2}
\end{equation}
and
\begin{equation}
\label{evalSgn}
%
\Ll\left(\frac{x}{2},0\right)=e^{-ix/2}L\left(x,-\frac{1}{2}\right)=
\log\left|\cot\frac{x}{4}\right|+i \frac{\pi}{2}\,\sign x.
\end{equation}

Another special case is $S(\pi,\mu)=0$ and $\Tl(\pi/2,\lambda)=0$ for any $\mu$, $\lambda$.
Introduce notation for the special values 
\begin{equation}
\label{Spi}
\Spi(\lambda)=\Sl\left(\frac{\pi}{2},\lambda\right)=-T\left(\pi,\frac{\lambda-1}{2}\right)
=2\sum_{k=1}^\infty\frac{(-1)^{k-1}}{2k-1+\lambda}.
\end{equation}
The function $\Spi(\lambda)$ can be expressed also in terms of the digamma function 
$$
 \Spi(\lambda)=\frac{1}{2}\left(\psi\left(\frac{\lambda+3}{4}\right)-\psi\left(\frac{\lambda+1}{4}\right)\right).
$$

\subsection{Laplace transform representation of the FJ sums}

\begin{proposition}
\label{prop:lt-L}
The identity
\begin{equation}
\label{lt:L}
L(x,\mu)=\int_0^\infty \frac{e^{-\mu u}\,du}{e^{u-ix}-1}
\qquad (\mu>-1, \;x\notin 2\pi\ZZ)
\end{equation}
holds.
Equivalently, for $\lambda>-1$
 \begin{equation}
\label{lt:Ll}
\Ll(x,\lambda)=L\left(2x,\frac{\lambda-1}{2}\right)e^{-ix}
=\int_0^\infty  \frac{e^{-\lambda u}\,du}{\sinh (u-ix)}.
\end{equation}
In particular,
\begin{equation}
\label{lt:Spi}
 \Spi(\lambda)=\int_0^\infty\frac{e^{-\lambda u}}{\cosh u}\,du.
\end{equation}
\end{proposition}

\begin{proof}
Eq. \eqref{lt:L} is obtained by substituting $(k+\mu)^{-1}=\int_0^\infty e^{-(k+\mu)u}\,du$ in the definition 
\eqref{Lmu} and summing the obtained geometric series under
the integral sign.
\end{proof}

\begin{remark*}
The function $L(x,\mu)$ is a particular case of the
{Lerch Zeta function} 
$$
 \zeta(s,t,\mu)=\sum_{k=0}^\infty \frac{e^{2\pi i m t}}{(k+\mu)^s}.
$$
Namely,
\[
L(x,\mu)= \zeta(1,x/2\pi,\mu).
\] 
Correspondingly, Eq. \eqref{lt:L} is a particular case of the integral representaion 
\cite[(25.14.5)]{NIST}
$$
\zeta(s,t,\mu)=\frac{1}{\Gamma(s)}\int_0^\infty \frac{u^{s-1} e^{-\mu u}}{1-e^{2\pi i t-u}}\,du.
$$
\end{remark*}

An integral representation equivalent to the imaginary part of \eqref{lt:L} is used in \cite{Fikioris-Andrianesis2015}
to explore the asymptotics of $S(x,0)-S_n(x,0)$ in the regimes $n\to\infty$, $x$ fixed and $n\to\infty$, $nx$ fixed. 

Equivalent representations in the form of Mellin-type integrals for $S_n(x,0)$ and $T_n(x,0)$ were
employed much earlier by Nikonov \cite{Nikonov1939} to describe the enveloping curve for extrema of the functions
$S_n(\cdot,0)$ and $T_n(\cdot,0)$.

Some further versions of the representations \eqref{lt:L}, \eqref{lt:Ll} will be useful.

Making the change of variable $z=u-ix$, we can rewrite \eqref{lt:L} so as to obtain the dependence on $x$ only in the
limit of integration,
\begin{equation}
\label{lt:L-ix}
e^{i\mu x}\,L(x,\mu)=\int_{-ix}^{-ix+\infty} \frac{e^{-\mu z}\,du}{e^{z}-1}.
\end{equation}
Similarly from \eqref{lt:Ll} we derive the representations
\begin{eqnarray}
\label{lt:Ll-ix}
e^{i\lambda x}\,\Ll(x,\lambda)&=&\int_{-ix}^{+\infty} \frac{e^{-\lambda z}\,dz}{\sinh z}
\\[2ex]
\label{lt:Ll-cosh}
&=&\pm i\,e^{\pm i\frac{\pi}{2}\lambda}\,\int_{-ix \pm i \frac{\pi}{2}}^{+\infty} \frac{e^{-\lambda z}\,dz}{\cosh z}.
\end{eqnarray}
These expressions are well suited for differentiation with respect to $x$.

\subsection{Integrated Dirichlet kernels}

We interpret the standard Dirichlet kernel
$$
 D_n(x)=1+2\sum_{k=1}^{n}\cos kx=\frac{\sin(n+\frac{1}{2})x}{\sin(x/2)} 
 \qquad (n=1,2,\dots)
$$
as a member of the family of functions $t\mapsto \sin(\lambda t)/\sin t$ with continuous parameter $\lambda>0$.

The identity 
$$
S_n'(x,0)=\frac{1}{2}(D_n(x)-1)
$$
 implies an expression of $S_n(x,0)$ 
in terms of the integrated Dirichlet kernel
\begin{equation}
\label{Sn-Dint}
 S_n(x,0)=\frac{1}{2}\int_0^x \left(D_n(t)-1\right)\,dt 
\end{equation}
or, equivalently,
\begin{equation}
\label{Sn-Dcint}
 S_n(x,0)=\int_0^{(\pi-x)/2} \frac{\cos(2n+1)t}{\cos t}\,dt.
\end{equation}

We introduce a family of functions 
\begin{equation}
\label{Dint-ss}
\Ssi(x,\lambda)=\int_0^x \frac{\sin\lambda t}{\sin t}\,dt,
\end{equation}
which will be useful in the analysis of more general FJ sums.
The sine integral $\Si(t)$ in \eqref{si} is a limiting case (see \eqref{Si-Ssi} below), which explains our choice of notation. 

Let us introduce notation for similar integrals with cosine in the denominator:
the complex-valued integral
\begin{equation}
\label{Eci}
\Eci(x,\lambda)=\int_0^x \frac{e^{it\lambda}}{\cos t}\,dt=\Cci(x,\lambda)+i\,\Sci(x,\lambda)
\end{equation}
and the corresponding real-valued integrals
\begin{equation}
\label{Dint-sc}
\Sci(x,\lambda)=\int_0^x \frac{\sin\lambda t}{\cos t}\,dt = \Im\,\Eci(x,\lambda)
\end{equation}
and
\begin{equation}
\label{Dint-cc}
\Cci(x,\lambda)=\int_0^x \frac{\cos\lambda t}{\cos t}\,dt= \Re\,\Eci(x,\lambda).
\end{equation}

\subsection{Limit relations}

\begin{proposition}
\label{prop:limrel-Dint}
(a) The exponential 
integral \eqref{E} is the limit case
of FJ sums:
if $\nu=x\mu$ is fixed and $\mu \to \infty$, 
or if $\nu=y\lambda$ and $\lambda \to \infty$,
then
\begin{equation}
\label{L-E}
L(x,\mu)\to E(\nu),
\qquad
\Ll(y,\lambda)\to E(\nu)
.
\end{equation}


(b) The trigonometric integral $\Si$ is limit case of integrated Dirichlet kernels:
if $\nu=x\lambda$ is fixed and $\lambda\to \infty$, 
then
\begin{equation}
\label{Si-Ssi}
\Ssi(x,\lambda) \to
\Si(\nu),
\end{equation}
%
\end{proposition}

(c) Under the same assumptions, the limit relation for $\Eci(x,\lambda)$ is
\begin{equation}
\label{Eci-lim}
\lambda\,\Eci(x,\lambda)\to i(1-e^{i\nu}).
\end{equation}

\begin{proof}
(a) Re-scaling the variable of integration $u\mapsto \mu u$ in \eqref{lt:L}, we obtain the integrand that
tends pointwise to the integrand in \eqref{lt:E} as $x\mu=\const$ and $x\to 0$. The conditions of 
the dominated convergence theorem are obviously met and we get the first limit relation in \eqref{Si-Ssi},
hence the others. 

(c) Similarly, rescale the integration variable in \eqref{Eci}.

\smallskip
(b) If $x<(1+\eps)\sin x$, then clearly $\Si(x\lambda)<\Ssi(x,\lambda)<(1+\eps)\,\Si(x\lambda)$.
Thus \eqref{Si-Ssi} follows .

\end{proof}

As a remark to the limit relations \eqref{L-E} let us look at the differential equation 
satisfied by the exponential integral \eqref{E}, 
\begin{equation}
\label{difeq:E}
 E'(t)=-\int\frac{e^{-tu}\,u\,du}{u-i}=-iE(t)-\frac{1}{t}.
\end{equation}

Its closer analog among the two, involving either the function $L$ or $\Ll$,
is the latter, 
\begin{equation}
\label{difeq:Ll}
\Ll'_x(x,\lambda)
=-i\lambda \Ll(x,\lambda)-\frac{1}{\sin x}. 
\end{equation}
It  is most easily derived using Eq.~\eqref{lt:Ll-ix}.

\subsection{FJ sums and integrated Dirichlet kernels}

A motivation for formulas to be presented in this subsection
comes from the familiar relation between $D_n(x)$ and $S_n(x,0)$ taking \eqref{L-S} into account,
\begin{equation}
\label{tSn-Dint}
\Im (e^{ixn} L(x,n))=\frac{\pi-x}{2}-S_n(x)=\frac{\pi}{2}-\int_0^x\frac{\sin\left(\left(n+\frac{1}{2}\right)t\right)}{2\sin\frac{t}{2}}\,dt,
\end{equation}
and a  similar representation for the remainder of the cosine Fourier series $T(x,0)$.
Recall the  
Fej\'{e}r's identity \cite[p.886,Eq.(1.2.12)]{MilRas1991}
$$
\sum_{k=1}^n\sin kx-\frac{1}{2}\sin nx=\frac{1}{2}\cot\frac{x}{2}(1-\cos nx),
$$ 
which, in view of the evaluation \eqref{evalSaw} of $T(x)$, is equivalent to
$$
(T_n(x)-T(x))'=\frac{\cot\frac{x}{2}\cos nx-\sin nx}{2} =\frac{\cos\left(n+\frac{1}{2}\right)x}{2\sin\frac{x}{2}}.
$$
Integrating and taking \eqref{L-S} into account, we get
\begin{equation}
\label{tTn-Dint}
\Re(e^{ixn} T(x,n))=\int_x^\pi \frac{\cos\left(n+\frac{1}{2}\right)t}{2\sin\frac{t}{2}}\,dt \;+\,e^{i\pi n}\,T(\pi,n).
\end{equation}

The next proposition asserts that the identities \eqref{tSn-Dint} and \eqref{tTn-Dint} remain
valid in the case of a continuous parameter. 

\begin{proposition}
\label{prop:L-Eci}
For any $\mu>-1$ and $0<x<2\pi$ the following identities (with $x$-independent right-hand sides) hold:
\begin{equation}
\label{Im-eL}
\Im(e^{ix\mu}L(x,\mu))
+\int_0^{x}\frac{\sin(\mu+\frac{1}{2})y}{2\sin \frac{y}{2}}\,dy=\frac{\pi}{2}
\end{equation}
and
\begin{equation}
\ba{rl}
\label{Re-eL}
\dst
\Re(e^{ix\mu}L(x,\mu))
-\int_x^\pi\frac{\cos(\mu+\frac{1}{2})y\,dy}{2\sin\frac{y}{2}}
&\dst
=
\Re(e^{i\pi\mu}T(\pi,\mu))
\\[1ex] &\dst
=-\Spi(2\mu+1)\,\cos\pi\mu .
\ea
\end{equation}
\end{proposition}

\begin{proof}
From the integral representation \eqref{lt:L-ix} it follows that
$$
 \frac{d}{dx} (e^{ix\mu} L(x,\mu))
 =-\frac{e^{ix\left(\mu+\frac{1}{2}\right)}}{2\sin(x/2)}
$$
and it becomes obvious that the left-hand sides in  \eqref{Im-eL}, \eqref{Re-eL} have zero derivatives with respect to $x$.
It remains to evaluate the constants (w.r.to~$x$) in the right-hand sides.
For \eqref{Re-eL} we simply take $x=\pi$. 
For \eqref{Im-eL} we let $x\to 0^+$ and use the evaluation
$$
\lim_{x\to 0^+} S(x,\mu)=\frac{\pi}{2}.
$$
It can be obtained in different ways. Here is one. 

We assume that for $\mu=0$ the result is known. Now,
$$
\lim_{x\to 0^+} (S(x,\mu)-S(x,0))=\lim_{x\to 0^+}\sum_{k=1}^\infty\left(\frac{1}{k+\mu}-\frac{1}{k}\right)\sin kx
=
0,
$$
since the series converges uniformly in $x$.
\end{proof}

Next we will give an expression for the integrals with cosine in the denominator. Notice that the integral
in \eqref{Im-eL} can be written as $\int_0^x=\int_0^{\pi}-\int_x^\pi$, since there is no singularity at $x=\pi$.
Having done this, we combine \eqref{Im-eL} and \eqref{Re-eL} and conclude that the expression
$$
  e^{ix\mu} L(x,\mu)-\int_x^\pi \frac{e^{i\left(\mu+\frac{1}{2}\right)y}}{2\sin\frac{y}{2}}\,dy
$$
does not depend on $x$. Thus we come to the formula
$$
 \int_x^\pi \frac{e^{i\left(\mu+\frac{1}{2}\right)y}}{2\sin\frac{y}{2}}\,dy=
 e^{ix\mu} L(x,\mu)-e^{i\pi\mu} L(\pi,\mu).
$$
Setting $x=\pi-2\xi$, $y=\pi-2t$, and $\lambda=2\mu+1$, we get after some simplification
%
$$
\int_0^\xi \frac{i \,e^{-i t\lambda}}{\cos t}\,dt=
 e^{-2i\xi\mu} L(\pi-2\xi,\mu)-L(\pi,\mu).
$$

Playing with the obtained formula a little further, we come to the result stated in the next Proposition.
In the formulation we used the function $\Ll$ instead of $L$ and the notation \eqref{Spi}.
The integral representation comes from \eqref{lt:Ll-cosh}.

\begin{proposition}
\label{prop:Eci-L}
The function $\Eci(x,\lambda)$ defined by the integral \eqref{Eci} is related to the function \eqref{Llam}
as follows:
\begin{equation}
\label{prop:Eci-L}
\Eci(x,\lambda)
=e^{ix \lambda} \Ll\left(x-\frac{\pi}{2},\lambda\right)+
i \Spi(\lambda). 
\end{equation}
It admits the Laplace transform representation
\begin{equation}
\label{lt:Eci}
\Eci(y,\lambda)=i\int_0^\infty e^{-\lambda u}
 \left(\frac{1}{\cosh u}-
\frac{e^{i\lambda y}}{\cosh(u-i y)}\right)\,du.
\end{equation}
\end{proposition}


\section{Main inequalities}
\label{sec:main_ineq}

\subsection{Majorization of the FJ sums}
\label{ssec:majorL}

\begin{theorem} 
\label{thm:FJ-L-M} 
(a) For $0<x<\pi$ the inequalitiy
\begin{equation}
\label{ineq:L-M} 
|L(x,\mu)|<M\left((2\mu+1)\sin \frac{x}{2}\right)
\qquad (\mu>-1/2)
\end{equation}
and its equivalent
\begin{equation}
\label{ineq:Ll-M} 
|\Ll(x,\lambda)|<M\left(\lambda\sin x\right)
\qquad (\lambda>0)
\end{equation}
hold.

(b) If $0< x<\pi$ and $\lambda\sin x>t_0$, where $t_0\approx 0.71\dots$ is 
the threshold constant defined in Proposition~\ref{prop:E-si-acot-M}(b), 
then  the inequality
\begin{equation}
\label{ineq:F-acot}
 \left|\Ll(x,\lambda)\right|<\arccot(\lambda\sin x)
\end{equation}
holds.

(c) If $(2\mu+1)\sin \frac{x}{2}<1$, then 
\begin{equation}
\label{ineq:L-log}
 \left|L(x,\mu)\right|<\ln\frac{1}{(2\mu+1)\sin \frac{x}{2}}+C_2,
\end{equation}
where $C_2=\ln(1+\sqrt{2})\approx 0.88$ as in Proposition~\ref{prop:ubM-small-t}.

\end{theorem}

\begin{proof}
By Proposition~\ref{prop:E-si-acot-M}, (a) $\Rightarrow$ (b),
and 
by Proposition~\ref{prop:ubM-small-t}, (a) $\Rightarrow$ (c).

Proof of (a): 
$$
|\sinh(u-ix)|^2=\sinh^2 u+\sin^2 x>u^2+\sin^2 x,
$$
hence
$$
|\Ll(x,\lambda)|< \int_0^\infty\frac{e^{-u\lambda}\,du}{|\sinh(u-ix)|}
<\int_0^\infty\frac{e^{-u\lambda}\,du}{\sqrt{u^2+\sin^2 x}}
< M(\lambda\sin x).
$$
\end{proof}

\begin{theorem}
\label{thm:ImF-arccot}
For any $\mu\geq -1/2$ and any $x\in[0,\pi]$ the inequality
\begin{equation}
\label{ineq:ImF-arccot}
\left|\Im\, e^{ix\mu} L(x,\mu)\right|\leq \arccot\left((2\mu+1)\sin \frac{x}{2}\right)
\end{equation}
holds. The equality takes place if $\mu=-1/2$ or $x=0$. 
%
\end{theorem}

\begin{proof} 
If $(2\mu+1)\sin(x/2)\geq 1 (>t_0)$, then the result follows from part (b) of Theorem~\ref{thm:FJ-L-M}. 

Suppose now that $(2\mu+1)\sin(x/2)< 1$.
Then $(\mu+1/2)x<\pi/2$ and $(\mu+1)x<\pi$. 
By \eqref{lt:L-ix},
$$
\Im\,\left( e^{i x\mu} L(x,\mu)\right)=\int_0^\infty e^{-\mu u} \frac{\Im \left\{e^{i x\mu}(e^{u+i x}-1)\right\}\,du}{|e^{u-ix}-1|^2}.
$$
The numerator of the integrand is positive, since
$$
e^u\sin x(\mu+1)-\sin x\mu> \sin x(\mu+1)-\sin x\mu=2\sin\frac{x}{2}
 \,\cos(2\mu+1)\frac{x}{2}>0.
$$
Thus we have the lower bound $\Im\, e^{i x\mu} L(x,\mu)>0$ in this case.

On the other hand, using the representation \eqref{Im-eL} and the estimate of Proposition~\ref{prop:intsinsin-Si} 
 (which is placed in Subsection~\ref{ssec:trig_ineq} for subject relevance)
 we get the upper bound
$$
\Im\, \left(e^{ix\mu} L(x,\mu)\right)=
\frac{\pi}{2}-\int_0^{x}\frac{\sin(\mu+\frac{1}{2})y}{2\sin \frac{y}{2}}\,dy
<\frac{\pi}{2}-\Si\left((2\mu+1)\sin\frac{x}{2}\right).
$$
By \eqref{ineq:si-acot}, the left-hand side is bounded by 
$\arccot\left((2\mu+1)\sin \frac{x}{2}\right)$.
\end{proof}

\begin{remark*}
A slightly coarser but often sufficient and more practical form of the estimate 
\eqref{ineq:L-M}
is
\begin{equation}
\label{ineq:L-frac}
\left|L(x,\mu)\right|\leq \frac{1}{(2\mu+1)\sin \frac{x}{2}}.
\end{equation}
\end{remark*}

\subsection{Some integral trigonometric inequalities}
\label{ssec:trig_ineq}

We begin with an estimate used in the proof of Theorem~\ref{thm:ImF-arccot}.
Compare with limit relation \eqref{Si-Ssi}.

\begin{proposition}
\label{prop:intsinsin-Si}
If $\lambda>0$, $0<x<\pi$, and $x\lambda<\pi/2$, then
$$
\int_0^{x}\frac{\sin \lambda y}{\sin y}\,dy>\mathrm{Si}(\lambda\sin x)
$$
\end{proposition}


\begin{proof}
Denote $\nu=\lambda \sin x$, so $\nu'_x=\lambda\cos x$. Differentiating, we get
$$
\ba{rcl}
\dst
\frac{d}{dx}\left(\int_0^{x}\frac{\sin y\lambda}{\sin y}\,dy-\int_0^{\nu}\frac{\sin y}{y}\,dy\right)
&=&\dst
\frac{\sin x\lambda}{\sin x}-\frac{\lambda\sin\nu\,\cos x}{\nu}
\\[1.5ex] &=& \dst
\frac{\sin x\lambda-\sin\nu\,\cos x}{\sin x}>0,
\ea
$$
because $\sin x>0$ and 
$\sin \lambda x>\sin\nu>\sin\nu\,\cos x$. 
\end{proof}

The next theorem is a sharper version of one of the results featured in the Introduction. 

\begin{theorem}
\label{thm:ebycos}
For all $\lambda>0$ and all $0<x<\pi/2$ the inequality
\begin{equation}
\label{ebycos}
\left|\Eci(x,\lambda)
-\frac{i}{\lambda}\right|<\frac{1}{\lambda}-M(\lambda)+M(\lambda\cos x).
\end{equation} 
holds.
\end{theorem}

\begin{proof}
Using \eqref{lt:Eci} and substituting $1/\lambda=\int_0^\infty e^{-u\lambda}\,du$ 
we get 
$$
 i\Eci(y,\lambda)+\frac{1}{\lambda}=
 \int_0^\infty e^{-\lambda u}\left(-\frac{1}{\cosh u}+\frac{e^{i\lambda y}}{\cosh(u-iy)}+1\right)\,du
 =A\,+\,e^{i\lambda y}\,B,
$$
where
$$
\ba{l}
\dst
 A=\int_0^\infty e^{-\lambda u}\left(1-\frac{1}{\cosh u}\right)\,du,
 \qquad
 B=\int_0^\infty \frac{e^{-\lambda u}\,du}{\cosh(u-iy)}.
 \ea
$$
By the triangle inequality,
$$
 \left|\Eci(x,\lambda)
-\frac{i}{\lambda}\right|\leq A+|B|\leq 
\int_0^\infty e^{-\lambda u}\left(1-\frac{1}{\cosh u}+\frac{1}{\sqrt{\sinh^2 u+\cos^2 x}}\right)\,du.
$$
We will show that the last integral is less than the right-hand side of \eqref{ebycos}, which is represented by the integral
 $$
 \frac{1}{\lambda}-M(\lambda)+M(\lambda\cos x)=\int_0^\infty e^{-\lambda u}\left(1-\frac{1}{\sqrt{u^2+1}}+\frac{1}{\sqrt{u^2+\cos^2 x}}\right)\,du.
$$
This can be achieved by comparing the integrands pointwise. 
Put
$$
 f(a, b)=(a^2+b^2)^{-1/2}.
$$
The targeted inequality between the integrands is
$$
 1-f(\sinh u,1)+f(\sinh u, \cos x)< 1-f(u,1)+f(u,\cos x).
$$
Note that the function $v\mapsto f(v,a)-f(v,b)$ is positive and decreasing whenever $0<a<b$.
The required result follows if we take $a=\cos x$, $b=1$ and compare the differences for $v_1=u$
and $v_2=\sinh u>v_1$. 
\end{proof}

\setcounter{corollary}{0}
\begin{corollary}
For any $\lambda\in\RR$ and $0<x<\pi/2$
the inequality
\begin{equation}
\label{ebycos2}
 \left|\int_0^x \frac{e^{iy\lambda}}{\cos y}\,dy-\frac{i}{\lambda}\right|<\frac{1}{\lambda\cos x}.
\end{equation} 
holds.
\end{corollary}

\begin{proof}
By Proposition~\ref{prop:M-1}, $1/\lambda-M(\lambda) <1/(\lambda\cos x)-M(\lambda\cos x)$
hence the right-hand side in \eqref{ebycos} is majorized by $1/(\lambda\cos x)$.
\end{proof}

The inequality \eqref{ebycos2} is illustrated in Figure~\ref{fig:trig-ineq}.

\begin{remark*}
The right-hand side of \eqref{ebycos} cannot be in general replaced by $M(\lambda\cos x)$
by analogy with \eqref{ineq:Ll-M}. However, the estimate \eqref{ebycos2} is 
about as good as \eqref{ebycos} except in the region where $t\cos x$ is small. 
\end{remark*}

\begin{figure}
\begin{center}
\begin{picture}(250,250)
{\small 
\put(0,0){\includegraphics[trim=2.5cm 1cm 0cm 1.2cm, width=250pt]{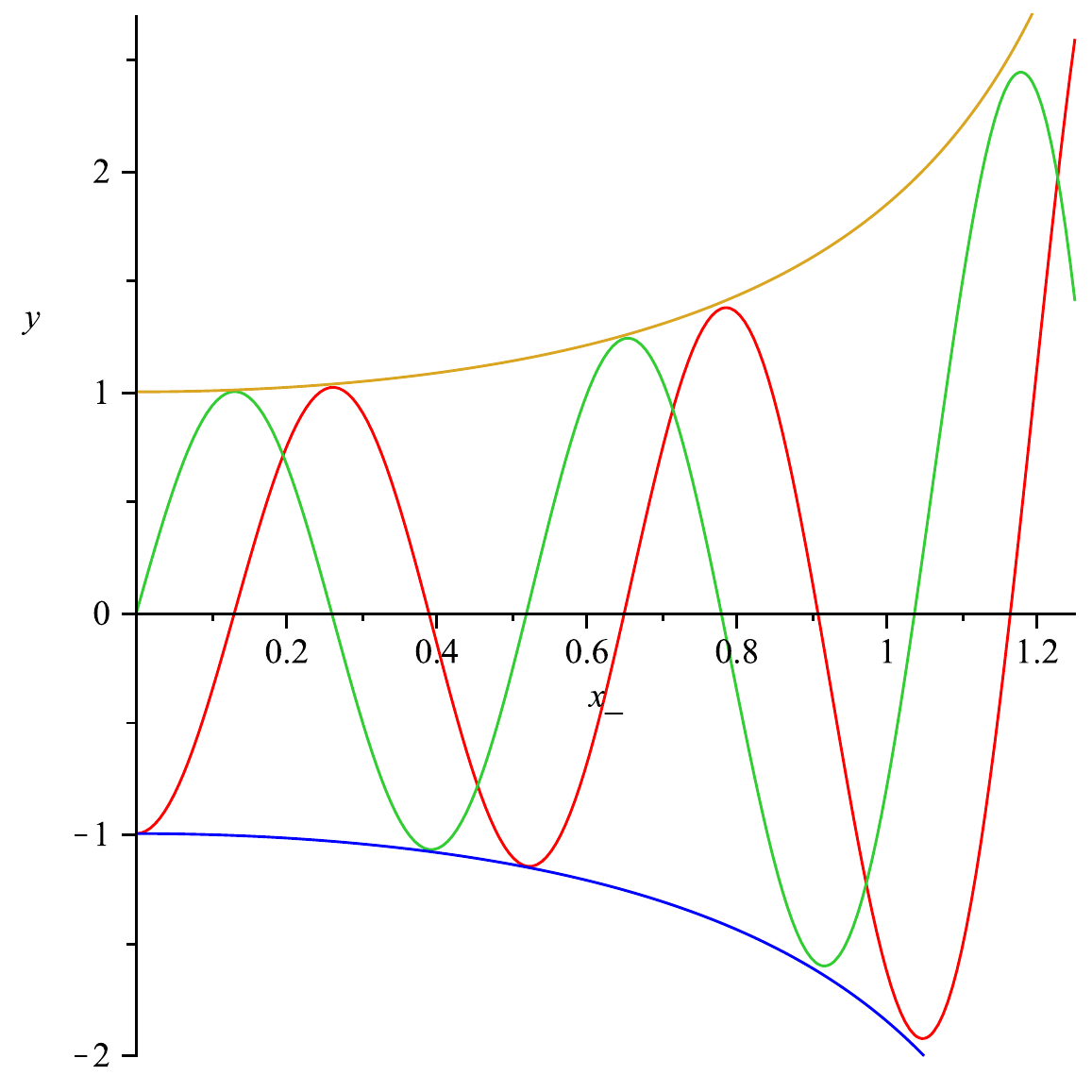}}
\put(72,180){$\sec x$}
\put(74,32){$-\sec x$}
}
\end{picture}
\end{center}
\caption{The functions $\lambda \Cci(x,=\lambda)$ (green), $\lambda \Sci(x,\lambda)-1$ (red),
 and their envelopes $\pm \sec x$ for $\lambda=12$}
\label{fig:trig-ineq}
\end{figure}

Theorem~\ref{thm:intD-sumsquares}
is but a re-formulation of the above Corollary.

Another interesting reformulation of Corollary~1 involving real-valued functions is the following

\begin{corollary}
For any real $\lambda$ and $\theta$ and $0<x<\pi/2$
the inequality
\begin{equation}
\label{ebycos3}
\left|\lambda\int_0^x\frac{\cos(\lambda y-\theta)}{\cos y}\,dy-\sin\theta\right|<\frac{1}{\cos x}. 
\end{equation} 
holds.
\end{corollary}

\begin{proof}
Apply \eqref{ebycos2} to estimate the real part of the expression
$$
e^{-iy\theta}\left(\lambda\int_0^x \frac{e^{i y\lambda}}{\cos y}\,dy-i\right)
$$
\end{proof}


\begin{remark*}
Let us show that an analog of Corollary 2 with numerator $2$ (instead of 1) in
the right-hand side can be obtained easily. 

Since $1/\cos y$ is an increasing function, by the Second Mean Value Theorem for integrals there exists
$a\in (0,x)$ such that
$$
\ba{rcl}
\dst
\lambda \int_0^x \frac{\cos (\lambda y-\theta)}{\cos y}\,dy
&=&\dst
\int_0^a \lambda \cos (\lambda y-\theta)\,dy+
\frac{1}{\cos x}\int_a^x  \lambda \cos (\lambda y-\theta)\,dy
\\[2ex]
&=&\dst
\sin (\lambda a-\theta)+\sin\theta+\frac{\sin (\lambda x-\theta)-\sin (\lambda a-\theta)}{\cos x}
\ea
$$
Therefore
$$
\lambda\int_0^a \frac{\cos (\lambda y-\theta)}{\cos y}\,dy-\sin\theta
=\frac{\sin (\lambda x-\theta)-(1-\cos x)\sin (\lambda a-\theta)}{\cos x},
$$
where the absolute value of the numerator is clearly bounded by 2.
\end{remark*}

\section{The Fej\'{e}r-Jackson-Tur\'{a}n inequality and variants}
\label{sec:fjt}

\subsection{The main result}
The main result featured in Introduction is Theorem~\ref{thm:myFJT}. 
In notation introduced in subsection~\ref{ssec:def-FJsums}
it states that  for any $n=1,2,\dots$ and $0\leq x<\pi$
$$
 |S(x,0)-S_n(x,0)|\leq \arccot\left((2n+1)\sin \frac{x}{2}\right).
$$

All the necessary work has been done above and this inequality comes now as 
a particular case of Theorem~\ref{thm:ImF-arccot} with $\mu=n$.

Below, the bound \eqref{ineq:myFJT} is compared to bounds known from earlier papers
and a chain of reasoning that led to its discovery is outlined.

\subsection{Earlier bounds for Fej\'{e}r-Jackson sums}

The inequality $S_n(x,0)\geq 0$ ($0\leq x\leq\pi$) was conjectured by  Fej\'{e}r in 1910 and first
proved independently by  Jackson \cite{Jackson1911} and Gronwall \cite{Gronwall1912}. The 
chronology of the initial communications is
briefly described in \cite[p.~302]{MMR1994}. That result gave rise to a large body of literature on conditions for positivity in various classes of trigonometric and algebraic polynomials. Details can be found in the handbooks \cite[Ch.XXI]{Mitrinovic1993}, \cite[\S~3.14]{Finch2003}, in introductory sections of the papers cited below, 
and in references therein.
Some relatively recent works are  
\cite{Mondal-Swaminathan2011}, \cite{Alzer-Fuglede2012}, \cite{Kwong2015}, \cite{Alzer-Kwong2015}. 

Tur\'{a}n \cite{Turan1938} contributed an upper bound and wrote the two-sided estimate 
\begin{equation}
\label{FJT}
\left|\sum_{k=1}^n\frac{\sin kx}{k}-\frac{\pi-x}{2}\right|\leq
\frac{\pi-x}{2}=S(x,0),
\qquad 0\leq x\leq\pi.
\end{equation}
The attribute ``Fej\'{e}r-Jackson-Tur\'{a}n'' in the section title refers to this inequality.

One readily sees that the inequality \eqref{FJT} is quite loose except when $x$ is $O(1/n)$-close to $\pi$.
Attempts to improve it were undertaken primarily in the context of research on positivity 
and therefore the effort focused on improving the lower bound for $S_n(x,0)$. 

Fej\'{e}r proved \cite[Eq.(113)]{Fejer1928}, \cite[p.~314]{MMR1994} that
\begin{equation}
\label{lb:Fejer1928}
 S_n(x,0)\geq \frac{\sin x}{3}+\frac{\sin nx}{2n},
\end{equation}
where the right-hand side is positive for $\pi/n<x<\pi-\pi/n$, $n\geq 3$. 

Tur\'{a}n \cite{Turan1952} gave the $n$-independent lower bound valid for $n\geq 2$ and all $x\in (0,\pi)$ 
\begin{equation}
\label{lb:Turan1952}
S_n(x)>4\sin^2\frac{x}{2}\cdot \left(\cot\frac{x}{2}-\frac{\pi-x}{2}\right).  
\end{equation}
A few decades later it was improved by Alzer and Koumandos \cite{Alzer-Koumandos2003b}: 
\begin{equation}
\label{lb:AK2003}
 S_n(x,0)>x^2\left(\cot\frac{x}{2}-\frac{\pi-x}{2}\right). 
\end{equation}

Tur\'{a}n's upper bound $S_n(x,0)\leq \pi-x$ was improved (for even $n$) by
Alzer and Koumandos \cite{Alzer-Koumandos2003a}: if $n\geq 2$ is even, then 
$$
 S_n(x,0) <\alpha (\pi-x),
$$
where $\alpha\approx 0.66395$ is the optimal value.

Brown and Koumandos \cite{BrownKoumandos1998} proved that
\begin{equation}
\label{lb:BK1998}
 S_n(x,0)>\frac{1-\sin\frac{x}{2}}{\cos\frac{x}{2}} 
\end{equation}
provided $x$ is at least $O(1/n)$ (given by specific formulas) distance away from the endpoints of 
$[0,\pi]$ (if $n$ is odd, the endpoint $\pi$ needs not be separated).

These latter bounds are termed ``functional'' bounds meaning that their right-hand sides are independent of $n$.

In \cite[p.~393]{Koumandos2012} Koumandos gave the polynomial functional lower bound
$$
 S_n(x,0)>x\left(1-\frac{x}{\pi}\right)^3
$$
as a special case of an inequality with multiple parameters. 
But it is pointwise weaker than both \eqref{lb:AK2003} and \eqref{lb:BK1998}.


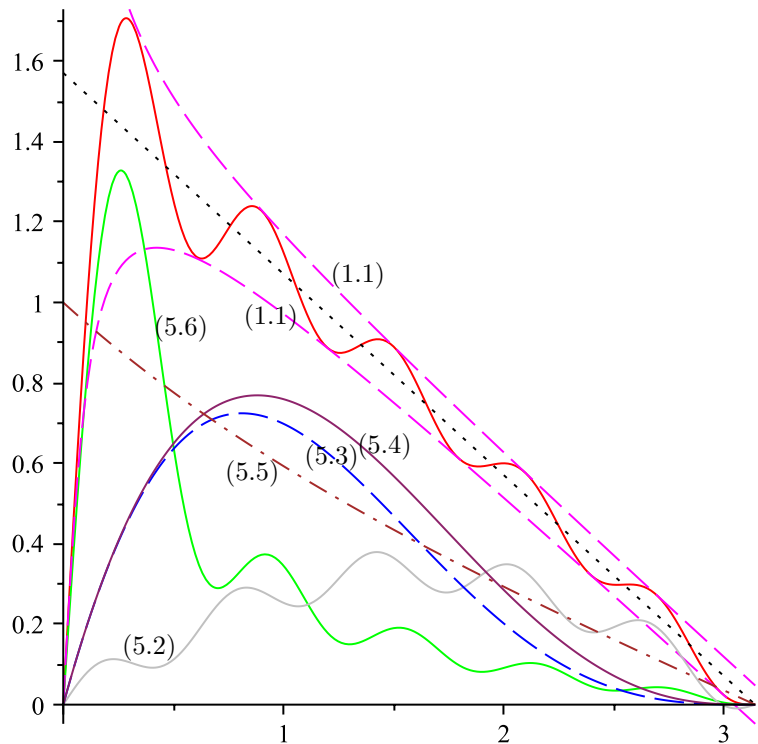
\begin{figure}
\begin{center}
\begin{picture}(350,260)
{\small
\put(0,-10){\includegraphics[clip, trim=0pt 1cm 0cm 0cm, width=300pt]{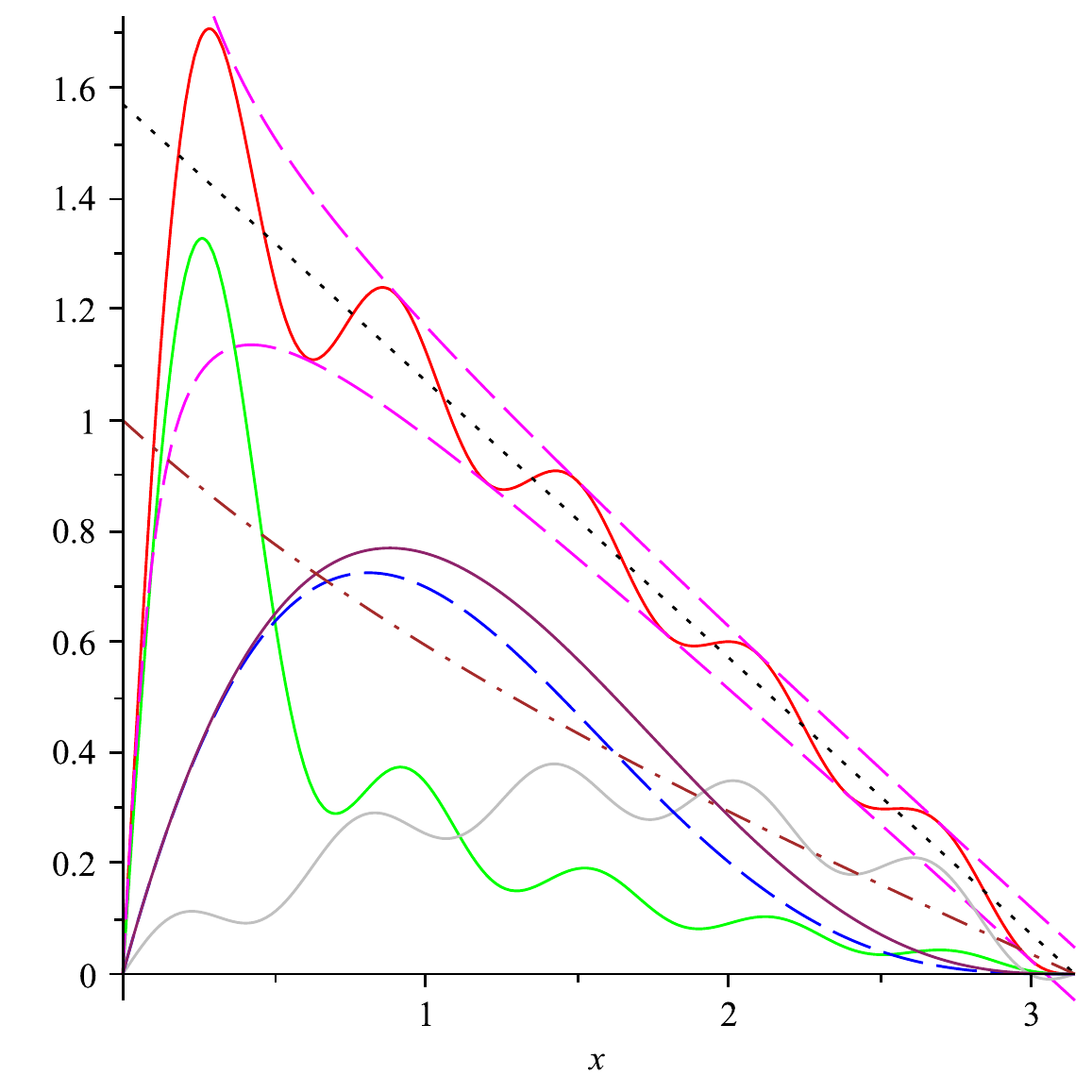}}
\put(135,167){\eqref{ineq:myFJT}}
\put(102,151){\eqref{ineq:myFJT}}
\put(56,26){\eqref{lb:Fejer1928}}
\put(125,98){\eqref{lb:Turan1952}}
\put(145,102){\eqref{lb:AK2003}}
\put(68,147){\eqref{AlKou12}}
\put(95,92){\eqref{lb:BK1998}}
}
\end{picture}
\end{center}
\caption{The function $S_{10}(x,0)$, its known minorants and the envelopes \eqref{ineq:myFJT}}
\label{fig:fejer-jackson-all-minorants-n10}
\end{figure}

Functional bounds, as interesting as they may be, cannot pretend to prop up the graph of $S_n(x,0)$ tightly
as  $n$ increases.
Alzer and Koumandos \cite{Alzer-Koumandos2012} 
gave a complicated $n$-dependent lower bound 
\begin{equation}
\label{AlKou12}
 S_n(x,0)\geq \frac{\pi}{4}\delta_n \cot\frac{x}{2} (1-P_n(\cos x)),
\end{equation}
where $P_n(x)$ are the Legendre polynomials and 
$$
 \delta_n=\frac{m+1}{m+3/2}\left(\frac{m!}{\Gamma(m+3/2)}\right)^2,
 \qquad m=\left\lfloor\frac{n-1}{2}\right\rfloor.
$$

Figure~1 illustrates, in the case $n=10$, the lower bounds \eqref{lb:Fejer1928}, \eqref{lb:Turan1952}, \eqref{lb:AK2003},
\eqref{lb:BK1998}, \eqref{AlKou12}, and our two-sided bound \eqref{ineq:myFJT}.  

The estimate \eqref{ineq:myFJT} faithfully reflects the behavior of the sums $S_n(x,0)$ in the ``main run'',
that is, away from small neighborhoods of the endpoints of size $O(1/n)$. Near the left end, \eqref{ineq:myFJT}
is slightly inferior to \eqref{AlKou12}. A neighbourhood of $\pi$ is the only region where \eqref{ineq:myFJT}
is not adequate: the bound has order of magnitude $O(1/n)$ comparable with magnitude of the estimated
function in that region.

\subsection{The genesis of the improved FJT inequality}


The Fej\'{e}r-Jackson inequality estimates the absolute value of the difference $S(x,0)-S_n(x,0)$. 
Classical formulas give the expression of this difference in terms of an integral of the $n$-th Dirichlet kernel,
$$
 S(x,0)-S_n(x,0)=\frac{\pi-x}{2}-\frac{1}{2}\int_0^x(D_n(t)-1)\,dt=\frac{\pi}{2}-
 \frac{1}{2}\int_0^x \frac{\sin \left(n+\frac{1}{2}\right)t}{\sin\frac{t}{2}}\,dt.
$$
Since $S(x,\pi)=S_n(x,\pi)=0$, it is possible to write the result in a ``pure'' integral form
$$
S(x,0)-S_n(x,0)=\frac{1}{2}\int_x^\pi \frac{\sin \left(n+\frac{1}{2}\right)t}{\sin\frac{t}{2}}\,dt,
$$
or, re-denoting the old $x$ as $\pi-2 x$ (with new $x$), putting
$$
 \lambda=2n+1
$$
for conciseness, and making an obvious change of variable, --- in the form
$$
S\left(\pi-2x,0\right)-S_n\left(\pi-2x,0\right)=(-1)^n\int_0^x \frac{\cos \lambda t}{\cos t}\,dt.
$$

Next, the integral in the right-hand  side is easy to bound by the value $2/(\lambda\cos x)$
as is done in Remark at the end of subsection~\ref{ssec:trig_ineq}.
One can search for a possible improvement of that bound.
Once the better bound $1/(\lambda\cos x)$ is discovered, one immediately obtains the weaker form
of the inequality \eqref{ineq:myFJT}
$$
\left|\sum_{k=1}^n\frac{\sin kx}{k}-\frac{\pi-x}{2}\right|<\frac{2}{(2n+1)x},
$$
which already supplies very adequate enveloping curves for the graph of $S(x,0)-S_n(x,0)$ except for small values of $x$.
Attempts to find an enveloping curve bounded near $0$ led to the discovery of the arccotangent bound of Theorem~\ref{thm:myFJT}.

A very different path that can potentially lead to the discussed tight envelopes might begin with
Nikonov's \cite{Nikonov1939} analysis of the curves containing minima and maxima of the sums $T_n(x,0)$
and $S_n(x,0)$. 

\section{Error of the Taylor polynomial approximation of the logarithmic function in the unit circle}
\label{sec:LogTaylor}

In this section we estimate the remainder 
$$
R_n(z)=\log(1-z)-\sum_{k=1}^n \frac{z^k}{k}=\sum_{k=n+1}^\infty \frac{z^k}{k}
$$
of Taylor's approximation of the logarithmic
function 
in the region of convergence. Our goal is to obtain an estimate that truly reflects a logarithmic (as opposed to a stronger one) character of the singularity.

To begin with, consider the case of real $z$ in the interval $(0,1)$
Then 
$$
 |R_n(z)|<\frac{1}{n+1}\sum_{k=n+1}|z|^k<\frac{|z|^{n+1}}{(n+1)|1-z|)}.
$$
This estimate does not hold for all $z$ in the region of convergence.
For instance, the asymptotics  at $z=-1$
$$
 \sum_{k=n+1}^\infty\frac{(-1)^{k-n-1}}{k}=\frac{1}{2n}-\frac{1}{4n^2}+O(n^{-3})=\frac{1}{2(n+1/2)}+O(n^{-3})
$$
shows that the inequality
$$
|R_n(z)|<\frac{1}{n+1}\sum_{k=n+1}|z|^k<\frac{|z|^{n+1}}{(n+\alpha)|1-z|}.
$$
cannot be true for $z=-1$ if $\alpha>1/2$.

The critical value $\alpha=1/2$ is already suitable. We will derive this result from an estimate for an FJ sum 
obtained in Section~\ref{sec:main_ineq}.

\begin{theorem}
\label{thm:Logtaylor-simple}
If $|z|\leq 1$, $z\neq 1$, then
\begin{equation}
\label{ineq:LogTaylor}
|R_n(z)|<\frac{|z|^{n+1}}{\left(n+\frac{1}{2}\right)|1-z|}.
\end{equation}
\end{theorem}

\begin{proof}
Conside the analytic function
$$
f(z)=(1-z)\sum_{k=1}^\infty \frac{z^k}{k+n} \,=\frac{1-z}{z^{n+1}}\,R_n(z).
$$ 
Take $z=e^{i\theta}$ on the boundary of the convergence circle. Then
$$
 f(e^{i\theta})=-2i e^{i\theta/2}\,\sin\frac{\theta}{2}\,\cdot\, L(\theta,n).
$$
By the inequality \eqref{ineq:L-frac}, which is a coarser form of the estimate in Theorem~\ref{thm:FJ-L-M},
we have
$$
 | f(e^{i\theta})|\leq \frac{2}{2n+1}, \quad -\pi<\theta<\pi.
$$
Applying the Maximum Modulus Principle to $f(z)$ in the unit circle we get the
claimed estimate.
\end{proof}

When $z$ is close to 1, the estimate given by Theorem~\ref{thm:Logtaylor-simple} is weak, since the series
in the left-hand side has logarithmic singularity at $z=1$, while the majorizing function grows as $|1-z|^{-1}$.
As a partial remedy, we will first derive a preliminary, relatively simple result, where the majorant
will behave as $O(\log(1-|z|))$ as $|z|\to 1^-$ (still not as $O(\log|1-z|)$).

\begin{proposition}
\label{prop:Logtaylor-log}
For $z$ in the unit circle put
$p=(n+1)(1-|z|)$ and 
$$
m(p)=\begin{cases}
(e\,p)^{-1}\;\;\text{if $p\geq e^{-1}$},\\
|\log p|\;\;\text{if $p<e^{-1}$}.
\end{cases}
$$
Then
\begin{equation}
\label{ineq:logtaylor-series-e}
 |R_n(z)|<m(p).
\end{equation}
\end{proposition}

\begin{proof}
Let $|z|=r<1$. Grouping the terms of the series into blocks of size $n+1$, we write
$$
 \sum_{k=n+1}^\infty \frac{z^k}{k}=\sum_{m=1}^\infty \sum_{j=0}^n \frac{z^{m(n+1)+j}}{m(n+1)+j}.
$$
Therefore
$$
|R_n(z)|
<\sum_{m=1}^\infty\frac{r^{m(n+1)}}{m(n+1)}
\sum_{k=0}^n r^k=\frac{-\ln(1-r^{n+1})}{n+1}\cdot\frac{1-r^{n+1}}{1-r}.
$$
Since $\max\limits_{0<x\leq 1} x|\ln x|=e^{-1}$, we get the unconditional estimate
$$
|R_n(z)|<(ep)^{-1}.
$$

Suppose now that $p<e^{-1}$. Substituting $t=1-r$ in the Bernoulli inequality
$(1-t)^{n+1}>1-(n+1)t$, we get $r^{n+1}>1-p$, hence $1-r{n+1}<p<e^{-1}$.

The 
function $x\mapsto x|\ln x|$ is increasing in $(0,e^{-1})$, so
$$
 (1-r^{n+1})|\ln (1-r^{n+1})|< p|\ln p|,
$$
 and the estimate
$
|R_n(z)|
<|\ln p|
$
follows.
\end{proof}

Theorem~\ref{thm:Logtaylor-simple} clearly beats Proposition~\ref{prop:Logtaylor-log} when $1-|z|$ is small
but $|1-z|$ is not.
This motivates us to look for an estimate involving $\log(|1-z| n)$
rather than $\log(1-|z|)n$. Such an estimate is presented in the next theorem.

\begin{theorem}
\label{thm:Logtaylor:log1-z}
If $|z|<1$ and $|1-z|(n+1/2)<1$, then 
\begin{equation}
\label{ineq:logtaylor:log1-z}
\left|\sum_{k=n+1}^\infty\frac{z^k}{k}\right|< 
|z|^{n+1}\left(-\log\left(|1-z|\left(n+\frac{1}{2}\right)\right)+C_3\right),
\end{equation}
where $C_3=\sqrt{(\pi/2)^2+1}\approx 1.862$. 
\end{theorem}

\begin{proof}
Put $\rho=(n+1/2)^{-1}$.
Consider the analytic function
$$
 f(z)=\frac{\dst z^{-(n+1)}R_n(z)}{-\log((1-z)(n+1/2))+1}
$$
in the domain 
$$
D_\rho=\{z:\, |z-1|<\rho,\;|z|<1\}.
$$
Note that $\Re(-\log (1-z)(n+1/2))\geq 0$ in $D_\rho$.

\smallskip
Let us estimate $|f(z)|$ on the boundary of $D_\rho$.

\smallskip
(i) On the part of the boundary of $D_\rho$ where $|z|=1$ we write $z=e^{i\theta}$ and use \eqref{ineq:L-log}
to obtain
$$
 |f(e^{i\theta})|<\frac{|L(\theta,n)|}{-\log ((2n+1)\sin\frac{\theta}{2})+1}\leq 1.
$$
(We take into account that $C_2<1$ in \eqref{ineq:L-log}.)

The bound on the unit circle seems rather weak.
However, the ``bad'' factor (with logarithmic growth near $\theta=0$) will get multiplied by $|z|^n$
and ensure a meaningful estimate inside the circle.

\smallskip
(ii)
On the part of the boundary of $D_\rho$ where $|1-z|=\rho$ we have
by Theorem~\ref{thm:Logtaylor-simple},
$$
 |f(z)|\leq|R_n(z)|\leq \frac{1}{(n+1/2)\rho}=1. 
$$

\smallskip
If the function $f(z)$ were continuous up to the boundary in $D_\rho$, then by the Maximum Modulus Principle
we could claim that $|f(z)|\leq 1$ in $D_\rho$. Let us assume that this inequality is true and deduce the conclusion of
the Theorem.

Put $t=-\log|(1-z)(n+1/2)|$ and $\alpha=\arg(1-z)$.
In the domain $D_\rho$ we have $t\geq 0$, $|\alpha|<\pi/2$.
Applying the elementaty inequality $(t+1)^2+\alpha^2\leq (t+\sqrt{1+\alpha^2})^2$,
we get
$$
 \left|-\log((1-z)(n+1/2))+1\right|<t+\sqrt{1+(\pi/2)^2}.
$$
Together with the inequality $|f(z)|\leq 1$ this leads to \eqref{ineq:logtaylor:log1-z}.

\smallskip
We return to justification of using the Maximum Modulus Principle. 
The problem is the presence of the essential singularity $z=1$ on the boundary of $D_\rho$. Let us recall that this fact cannot be ignored, cf.\ the function $\exp(1-z)^{-1}$,
which is defined and bounded on $\partial D_\rho$ except at $z=1$. We need to use the
Phragm\'{e}n-Lindel\"of theory, see e.g.\ \cite[Ch.~3]{Garding1997}.
The {\em critical growth condition}\ in this case (the boundary has a tangent at the exceptional point) is:
for any $\eps>0$
$$
 |f(z)|e^{-\eps/|z-1|}\to 0
$$
as $z\to 1$ inside $D_\rho$.
It is satisfied in our case due to the estimate \eqref{ineq:LogTaylor}.
\end{proof}

Theorems~\ref{thm:Logtaylor-simple} and \ref{thm:Logtaylor:log1-z} can be combined so as to obtain a formulation
analogous to that of Proposition~\ref{prop:Logtaylor-log}:

\begin{corollary*}
For $z$ in the unit circle put
$q=(n+1/2)|1-z|$ and 
$$
\hat m(q)=\min\left(q^{-1},\; \log q^{-1}+\sqrt{(\pi/2)^2+1}\right).
$$
Then
\begin{equation}
\label{ineq:logtaylor-series-e}
 |R_n(z)|<|z|^{n+1}\,\hat m(q).
\end{equation}
\end{corollary*}




\begin{thebibliography}{99}
\bibitem{Abramowitz-Stegun1964}
M. Abramowitz, I. Stegun (eds.) Handbook of mathematical functions. Nat. Bureau of Standards, N.-Y., 1964.


\bibitem{Alzer-Koumandos2003a}
H. Alzer, S. Koumandos, 
Sharp inequalities for trigonometric sums,
Mathematical Proceedings of the Cambridge Philosophical Society, Jan 2003, Vol.134(1), pp.139--152
https://doi.org/10.1017/S0305004102006357

\bibitem{Alzer-Koumandos2003b}
H. Alzer, S. Koumandos, Inequalities of Fej\'{e}r-Jackson type,
Monatsh. Math.139 (2003) 89--103.MR1987909 (2004d:26014)


\bibitem{Alzer-Fuglede2012}
H.Alzer, B. Fuglede,
On a trigonometric inequality of Tur\'{a}n,
Journal of Approximation Theory 164 (2012) 1496--1500.

\bibitem{Alzer-Koumandos2012}
H.Alzer, S.~Koumandos, 
Sharp estimates for various trigonometric sums,
Analysis 33 (2012) 9--26.

\bibitem{Alzer-Kwong2015}
H.Alzer, Man Kam Kwong,
Rogosinski-Szeg\"{o} type inequalities for
trigonometric sums,
Journal of Approximation Theory 190 (2015) 62--72.






\bibitem{BrownKoumandos1998}
G.~Brown, S.~Koumandos, A new bound for the Fej'{e}r-Jackson sum,
Acta Math. Hungar. 80, 1--2 (1998), 21--30.

\bibitem{Fejer1928}
L.~Fej\'{e}r, 
Einige S\"atze, die sich auf das Vorzeichen einer ganzen rationalen Funktion beziehen,
Monatsch. Math. Physik, 35 (1928), 305--344;
also in Gesammelte Arbeiten II, Birkh\"auser-Verlag, Basel and Stuttgart, 1970.  

\bibitem{Fikioris-Andrianesis2015} 
G. Fikioris, P. Andrianesis,
Asymptotic expansions pertaining to the logarithmic
series and related trigonometric sums,
J. Classical Analysis, 7:2 (2015), 113--127, 
doi:10.7153/jca-07-11. (files.ele-math.com/articles/jca-07-11.pdf)

\bibitem{Fikioris-Andrianesis2017} 
G. Fikioris, P. Andrianesis,
Asymptotic approximations elucidating the
Gibbs phenomenon and Fej\'{e}r averaging,
Asymptotic Analysis 102 (2017), 1--19;
DOI 10.3233/ASY-171408
 
\bibitem{Finch2003}
S. Finch, Mathematical constants, Cambridge Univ. Press, Cambridge, 2003.



\bibitem{Garding1997}
L.~G{\aa}rding, Some points of analysis and their history,
AMS, 1997. 

\bibitem{Gronwall1912}
T. H. Gronwall, 
\"Uber die Gibbssche Erscheinung und die trigonometrischen Summen 
$\sin x+\frac{1}{2}\sin 2x+\cdots+\frac{1}{n}\sin nx$,
Math. Ann. 72 (1912) 228--243.




\bibitem{Jackson1911} D. Jackson,
\"Uber eine trigonometrische Summe, Rend. Circ. Metem. Palermo, t.XXXII (1911), 257--262.



\bibitem{Koumandos2012}
S.~Koumandos,
Inequalities for trigonometric sums, in:
Nonlinear Analysis: Stability, Approximation, and Inequalities,
Springer, 2012, pp.387--416.


\bibitem{Kwong2015}
Man Kam Kwong,
An improved Vietoris sine inequality,
Journal of Approximation Theory 189 (2015)  29--42.


\bibitem{MMR1994}
G.V. Milovanovi\'{c}, D.S. Mitrinovic, Th.M. Rassias,
Topics in Polynomials: Extremal Problems, Inequalities, Zeros,
World Scientific, 1994.

\bibitem{MilRas1991}
G.V. Milovanovi\'{c}, Th.M. Rassias,
Inequalities connected with trigonometric sums,
in: Constantin Caratheodory: An International Tribute (Th.M. Rassias, ed.),
vol.~2, World Scientific Publ., 1991, 875--941.

\bibitem{Mitrinovic1993}
D. S. Mitrinovi\'{c}, J. E. Pe\v{c}ari\'{c}, A. M. Fink,
Classical and New Inequalities in Analysis,
Springer, Dordrecht, 1993.
https://doi.org/10.1007/978-94-017-1043-5

\bibitem{Mondal-Swaminathan2011}
S.R. Mondal, A. Swaminathan,
On the positivity of certain trigonometric sums and their applications,
Computers and Mathematics with Applications 62 (2011), 3871--3883.

\bibitem{Nikonov1939}
%
V. I. Nikonov,
The integral representation of certain trigonometrical polynomials as a method
of their investigation.
Trudy Leningradskogo  Industrial'nogo Instituta
Razdel Fiziko-matematicheskih nauk
v.3, no~1 (1939),
pp. 10--15. (In Russian)

\bibitem{NIST} F.~Olver, D.~Lozier, Boisvert, Clark (eds.)
NIST Handbook of Mathematical Functions, Cambridge Univ. Press, 2010.


\bibitem{Turan1938}
P.~Tur\'{a}n, \"Uber die Partialsummen der Fourierreihe, J. London Math. Soc., 13 (1938), 278--282. 

\bibitem{Turan1952}
P.~Tur\'{a}n,
On a trigonometrical sum. 
Ann. Soc. Polon. Math. 25 (1952), 155--161.












\end{thebibliography}
\end{document}